\documentclass[a4paper,leqno,11pt,reqno]{amsart}
\usepackage{epsf,graphicx,epsfig}
\usepackage{amscd}
\usepackage{chngcntr}
\usepackage{amsmath,latexsym,amssymb,amsthm}
\usepackage[nospace,noadjust]{cite}
\usepackage{textcomp}
\usepackage{setspace,cite}
\usepackage{lscape,fancyhdr,fancybox}
\usepackage[all,cmtip]{xy}
\counterwithout{equation}{section}

\usepackage{graphicx}

\usepackage{color}
\usepackage{url}
\usepackage{enumitem}
\usepackage[mathscr]{euscript}
  \theoremstyle{plain}

\swapnumbers
    \newtheorem{Thm}{Theorem}[section]
    \newtheorem{Prop}[Thm]{Proposition}
   \newtheorem{Lem}[Thm]{Lemma}
    \newtheorem{Cor}[Thm]{Corollary}
    
    \newtheorem{subsec}[Thm]{}
\theoremstyle{definition}
    \newtheorem{Def}[Thm]{Definition}
    \newtheorem{Exm}[Thm]{Example}

\theoremstyle{remark}
     \newtheorem{Rem}[Thm]{Remark}

\begin{document}
\title{Deformation of hom-Lie-Rinehart algebras}
\author{Ashis Mandal}
\address{Department of Mathematics and Statistics, Indian Institute of Technology,
Kanpur 208016, India.}
\author{Satyendra Kumar Mishra}
\address{Statistics and Mathematics Unit, Indian Statistical Institute Bangalore, 560059, India.}
\begin{abstract}
We study formal deformations of hom-Lie-Rinehart algebras. The associated deformation cohomology that controls deformations is constructed using multiderivations of hom-Lie-Rinehart algebras.
%In this paper we define a differential graded Lie algebra for hom-Lie-Rinehart algebras, which controls deformations of hom-Lie-Rinehart algebras.  
\end{abstract}
\maketitle
\section{Introduction}
The aim of this paper is to study formal deformations of hom-Lie-Rinehart algebras. In \cite{MM2018}, hom-Lie-Rinehart algebras are introduced as an algebraic analogue of hom-Lie algebroids. This notion generalises both the notion of a hom-Lie algebra and the notion of a hom-Lie algebroid in \cite{LGT13}. If we start with a Lie-Rinehart algebra (\cite{Rin63}, \cite{Hueb98}, \cite{Hueb99}, \cite{Hueb04}) and a homomorphism into itself, we obtain a canonical hom-Lie-Rinehart algebra structure(usually referred as ``obtained by composition").

In \cite{MM2018}, we study modules over a hom-Lie-Rinehart algebra and a cohomology with coefficients in a left module. We also characterize this cohomology (in low dimensions) in terms of the  group of automorphisms and the equivalence classes of abelian extensions (see \cite{MM2018} for more details). Later on, a non-abelian tensor product and universal central extensions  are introduced in \cite{MME18}, crossed modules for hom-Lie-Rinehart algebras are studied in \cite{TFY18}, and a relationship between hom-Lie-Rinehart algebras and hom-Gerstenhaber algebras is discussed in \cite{MM2017}. 

In the case of classical algebras, the associated cohomologies with coefficients in adjoint representations control deformations. For example, Hochschild (and Chevalley-Eilenberg) cohomology plays the role of deformation cohomology in the case of associative (and Lie) algebras. Likewise, for hom-algebras, in particular for hom-Lie algebras and hom-associative algebras, cohomologies with coefficients in certain adjoint representations are defined (generalizing Chevalley-Eilenberg and Hochschild cohomology, respectively) in \cite{AEM11}, which play the role of deformation cohomologies. 

For Lie algebroids, more work is required to find suitable deformation cohomology since the adjoint representation is not defined as representation in the usual sense (see \cite{TrMs}). In \cite{CM08}, M. Crainic and I. Moerdijk introduced the notion of multiderivations of a vector bundle. The space of multiderivations of a vector bundle forms a graded Lie algebra. If vector bundle has a Lie algebroid structure, a differential can be associated with this graded Lie algebra to obtain a differential graded Lie algebra (in short DGLA). They also proved that this DGLA controls deformations (in the sense of \cite{CM08}) of a Lie algebroid. In the present article, we describe a differential graded Lie algebra for hom-Lie-Rinehart algebra, which gives a deformation complex required to study one-parameter formal deformations.

In Section $2$, we recall some basic definitions. Let $R$ be a commutative ring of characteristic zero and with unity. Also, let $A$ be an associative commutative $R$-algebra, and $\phi:A\rightarrow A$ be an $R$-algebra homomorphism. 

In Section $3$, we first define the notion of $(\phi,\beta)$-multiderivations of a pair $(M,\beta)$, where $M$ is an $A$-module, and $\beta:M\rightarrow M$ is a $\phi$-function linear map. Then we consider a graded Lie algebra structure on the space of $(\phi,\beta)$-multiderivations.  Next, we describe hom-Lie-Rinehart algebra structures on the pair $(M,\beta)$ $($over $(A,\phi))$ in terms of this graded Lie algebra. Consequently, we associate a differential graded Lie algebra (DGLA) to a hom-Lie-Rinehart algebra. 

In Section $4$, we show that this DGLA controls the deformations of a hom-Lie-Rinehart algebra. We also prove that if $\phi: A\rightarrow A$ is an algebra automorphism, then the space of $\phi$-derivations of $A$ is rigid as a hom-Lie-Rinehart algebra. In the last section, we discuss  some particular cases of this deformation complex for hom-Lie-Rinehart algebras.

\section{Preliminaries}
In this section, we recall basic definitions concerning hom-structures from \cite{AEM11,HLS06,MS08,MS10,Sheng12,MM2018} in order to  fix notations and terminology. Let $R$ be a commutative ring of characteristic zero and with unity. Throughout the paper, we consider all modules, algebras and their tensor products over the ring $R$ and all linear maps to be $R$-linear unless otherwise stated. 

\begin{Def}
A hom-Lie algebra is a triple $\textstyle{(L,[-,-],\alpha)}$, which consists of a $R$-module $L$, a skew-symmetric $R$-bilinear map $\textstyle{[-,-]:L\times L\rightarrow L}$, and a linear map $\textstyle{\alpha:L\rightarrow L}$, satisfying
 $$[\alpha(x),[y,z]]+[\alpha(y),[z,x]]+[\alpha(z),[x,y]]=0 ~~~~\mbox{for all}~~x,y,z\in L.$$
 
 A hom-Lie algebra $\textstyle{(L,[-,-],\alpha)}$ is called multiplicative if the map $\alpha$ preserves the bracket, i.e. $\alpha[x,y]=[\alpha(x),\alpha(y)]$. Moreover, if $\alpha$ is an $R$-module automorphism of $L$, then the hom-Lie algebra $\textstyle{(L,[-,-],\alpha)}$ is called a regular hom-Lie algebra.
\end{Def}
\begin{Exm}
Given a Lie algebra $\textstyle{(L,[-,-])}$ with a Lie algebra homomorphism $\alpha:L\rightarrow L$, we can define a hom-Lie algebra as the triple $\textstyle{(L,\alpha\circ [-,-],\alpha )}$.
\end{Exm}
\begin{Def}\label{Rep-hom-Lie}
A representation of a hom-Lie algebra $\textstyle{(L,[-,-],\alpha)}$ on a $R$-module $V$ is a pair $(\theta, \beta)$ consisting of $R$-linear maps $\theta: \mathfrak{g}\rightarrow \mathfrak{g}l(V)$ and $\beta: V\rightarrow V$ such that 
\begin{align*}
\theta(\alpha(x))\circ \beta&= \beta\circ \theta(x),\\
\theta([x,y])\circ \beta&= \theta(\alpha(x))\circ\theta(y)-\theta(\alpha(y))\circ\theta(x).
\end{align*}
for all $x,y \in L$.
\end{Def}
\begin{Exm}
%\cite{Sheng}
 For $s\in \mathbb{Z}$, we can define the $\alpha^s$-adjoint representation of the regular hom-Lie algebra $(L,[-,-],\alpha)$ on $L$ by $(ad_s,\alpha)$, where
$$ad_s(g)(h)=[\alpha^s(g),h]~~~\mbox{for all}~g,h\in \mathfrak{g}.$$
\end{Exm}
Let $A$ be an associative commutative $R$-algebra and $Der(A)$ denote the space of $R$-derivations of the algebra $A$. Then $Der(A)$ is simultaneously an $A$-module and a Lie algebra with the commutator bracket. 
\begin{Def}\label{Lie-Rin}
A Lie-Rinehart algebra $L$ over (an associative commutative $R$-algebra) $A$ is a Lie algebra  over $R$ with an $A$-module structure and a $R$-module homomorphism $\rho:L\rightarrow Der(A)$, such that 
$\rho$ is simultaneously an $A$-module homomorphism and a Lie $R$-algebra homomorphism and 
$$ [x, ay]=a[x,y]+\rho(x)(a)y~~\mbox{ for all}~ x,y\in L,~a\in A.$$
 \end{Def}
\begin{Def}\label{DerS}
 Given an associative commutative algebra $A$, an $A$-module $M$ and an algebra endomorphism $\phi: A\rightarrow A$, we call an $R$-linear map $\delta: A\rightarrow M$ a $\phi$-derivation of $A$ into $M$ if it satisfies the required identity; 
$$\delta(a b)=\phi(a) \delta(b)+\phi(b) \delta(a)~~\mbox{for all}~~a,b \in A.$$
 Let us denote by $Der_{\phi}(A)$ the set of $\phi$-derivations.
\end{Def} 

Let $M$ be a smooth manifold, and $\psi: M\rightarrow M $ be a smooth map. Then the map $\psi$ induces an algebra homomorphism $\textstyle{\psi^*: C^{\infty}(M)\longrightarrow C^\infty(M)}$, and the space of $\psi^*$-derivations of $C^{\infty}(M)$ into itself can be identified with the space of sections of the pull- back bundle of the tangent bundle $TM$, which is denoted by $\textstyle{\Gamma(\psi^!TM)}.$ 

\begin{Def}\label{hom-Lie}
 A hom-Lie algebroid is a quintuple $(A,\phi,[-,-],\rho,\alpha)$, where $A$ is a vector bundle over a smooth manifold $M$, the map $\phi: M \rightarrow M$ is a smooth map, the bracket $[-,-]:\Gamma(A)\otimes \Gamma(A)\rightarrow \Gamma(A)$ is a bilinear map, the map $\rho: \phi^{!}A\rightarrow \phi^!TM$ is called the anchor and $\alpha: \Gamma (A)\rightarrow\Gamma(A)$ is a linear map such that following conditions are satisfied.
\begin{enumerate}[label=(\roman*)]
\item $\textstyle{\alpha(f.s)=\phi^*(f) \alpha(s)}$ for all $s\in \Gamma(A),~f \in C^{\infty}(M)$.
\item The triplet $(\Gamma(A),[-,-],\alpha)$ is a hom-Lie algebra.
\item The following hom-Leibniz identity holds:
$$\textstyle{[s,f.t]=\phi^*(f).[s,t]+\rho(s)[f]\alpha(t)};$$
for all $\textstyle{s,t\in \Gamma(A),~f \in C^{\infty}(M)}$.
\item The pair $(\rho,\phi^*)$ is a representation of $\textstyle{(\Gamma(A),[-,-],\alpha)}$ on $C^{\infty}(M)$.
\end{enumerate}

A hom-Lie algebroid $\textstyle{(A,\phi,[-,-],\rho,\alpha)}$ is called regular (or invertible) hom-Lie algebroid if the map $\alpha: \Gamma (A)\rightarrow\Gamma(A)$ is an invertible map and the smooth map $\phi: M \rightarrow M$ is a diffeomorphism.
\end{Def}

Here, $\rho(s)[f]$ denotes  a function on $M$ given by $$\rho(s)[f](m)=\big<d_{\phi(m)}f,\rho_m(s_{\phi(m)}) \big>$$ for $m\in M$. The map $\rho_m:(\phi^!A)_m\cong A_{\phi(m)}\rightarrow (\phi^!TM)_m\cong T_{\phi(m)}M$ is the anchor map evaluated at $m\in M$ and $s_{\phi(m)}$ is  image of the section $s\in\Gamma (A)$ at $\phi(m)\in M$.

\newpage

\subsection{Hom-Lie-Rinehart algebras:}
\begin{Def}\label{hom-LR}
A tuple $\textstyle{(A,L,[-,-]_L,\phi,\alpha_L,\rho_L)}$ is called a hom-Lie Rinehart algebra over $(A,\phi)$  if $L$ is an $A$-module, $\textstyle{[-,-]_L: L\times L\rightarrow L}$ is a skew-symmetric bilinear map, the map $\phi:A\rightarrow A$ is an algebra homomorphism, $\alpha_L :L\rightarrow L $ is a linear map preserving the bracket $[-,-]_L$, and $\rho_L: L\rightarrow Der_{\phi}A$ is a $R$-linear map satisfying the following conditions:
\begin{enumerate}[label=(\roman*)]
\item $\alpha_L(a.x)=\phi(a).\alpha_L(x)$ for all $a\in A,~x\in L $.
\item The triple $(L,[-,-]_L,\alpha_L)$ is a hom-Lie algebra.
\item The pair $(\rho_L, \phi)$ is a representation of $(L,[-,-]_L,\alpha_L)$ on $A$.
\item $\rho_L(a.x)=\phi(a).\rho_L(x)$ for all $a\in A,~x\in L $.
\item $[x,a.y]_L=\phi(a)[x,y]_L+ \rho_L(x)(a)\alpha_L(y)$ for all $a\in A,~x,y\in L $.
\end{enumerate}
\end{Def}

Let us denote a hom-Lie Rinehart algebra $\textstyle{(A,L,[-,-]_L,\phi,\alpha_L,\rho_L)}$ over $(A,\phi)$ simply by $(\mathcal{L},\alpha_L)$. In particular,

\begin{enumerate}[label=(\alph*)]
\item If $\alpha_L=Id_L$ in the above definition, then $\phi=id_A$ and the hom-Lie-Rinehart algebra $\textstyle{(A,L,[-,-]_L,\phi,\alpha_L,\rho_L)}$ is a Lie-Rinehart algebra $L$ over $A$.
 
\item Any hom-Lie algebra $(L,[-,-]_L,\alpha_L)$ (over $R$) is also a hom-Lie Rinehart algebra $(A,L,[-,-]_L,\phi,\alpha_L,\rho_L)$ with $A=R$, algebra morphism $\phi=Id_{R}$ and the trivial action of $L$ on $R$. 

\item Any hom-Lie algebroid $(A,\phi,[-,-],\rho,\alpha)$ over a smooth manifold $M$, gives a hom-Lie-Rinehart algebra $\textstyle{(C^{\infty}(M),\Gamma A,[-,-],\phi^*,\alpha,\rho)}$ over $(C^{\infty}(M),\phi^*)$, where $\Gamma A$ is the space a sections of the underline vector bundle $A$ over  $M$ and the algebra homomorphism $\textstyle{\phi^*:C^{\infty}(M)\rightarrow C^{\infty}(M)}$ is induced by the smooth map $\phi: M\rightarrow M$.
\end{enumerate}

\begin{Exm}\label{hom-LR byc}
If we consider a Lie-Rinehart algebra $L$ over $A$ along with an endomorphism $$(\phi,\alpha):(A,L)\rightarrow (A,L)$$ in the category of Lie-Rinehart algebras then the tuple $(A,L,[-,-]_{\alpha}, \phi,\alpha,\rho_{\phi})$ is a hom-Lie-Rinehart algebra, called ``obtained by composition", where
\begin{enumerate}[label=(\roman*)]
\item $[x,y]_{\alpha}=\alpha[x,y]$ for $x,y\in L$; 
\item $\rho_{\phi}(x)(a)=\phi(\rho(x)(a))$ for $x\in L, ~a\in A$.
\end{enumerate} 
\end{Exm}

\begin{Exm}\label{Product}
Let $(A,L,[-,-]_L,\phi,\alpha_L,\rho_L)$ and $(A,M,[-,-]_M,\phi,\alpha_M,\rho_M)$ be hom-Lie-Rinehart algebras over $(A,\phi)$. We consider $$L\times_{Der_{\phi}A}M=\{(l,m)\in L\times M : \rho_l(l)=\rho_M(m)\},$$ where $L\times M$ denotes the Cartesian product. Then $(A,L\times_{Der_{\phi}A}M,[-,-],\phi,\alpha,\rho)$ is a hom-Lie-Rinehart algebra, where  
\begin{enumerate}[label=(\roman*)]
\item the bracket is given by 
$$[(l_1,m_1),(l_2,m_2)]=([l_1,l_2],[m_1,m_2]);$$
\item the endomorphism $\alpha:L\times_{Der_{\phi}A}M\rightarrow L\times_{Der_{\phi}A}M$ is given by
$$\alpha(l,m)=(\alpha_L(l),\alpha_M(m));$$
\item and the anchor map $\rho: L\times_{Der_{\phi}A}M\rightarrow Der_{\phi}A$ is given by
$$\rho(l,m)(a)=\rho_L(a)=\rho_M(a);$$
\end{enumerate}   
for all $l,l_1,l_2\in L$, $m,m_1,m_2\in M$, and $a\in A$. The above  structure gives the categorical product in the category $hLR_A^{\phi}$. 
\end{Exm}

%\begin{Def}
%Let $(A,L,[-,-]_{L},\phi,\alpha_L,\rho_L)$ and $(B,L^{\prime},[-,-]_{L^{\prime}},\psi,\alpha_{L^{\prime}},\rho_{L^{\prime}})$ be hom-Lie-Rinehart algebras, then a {\bf homomorphism }of hom-Lie-Rinehart algebras is defined as a pair of maps $(g,f)$, where the map $g:A\rightarrow B$ is a $R$-algebra homomorphism and $f: L_1\rightarrow L_2$ is a $R$-linear map such that following identities hold:
%\begin{enumerate}
%\item $f(a.x)=g(a).f(x)~~\mbox{for all}~x\in L_1,~a\in A,$
%\item $f[x,y]_L=[f(x),f(y)]_{L^{\prime}} ~~\mbox{for all}~x,y\in L,$
%\item $f(\alpha_L(x))=\alpha_{L^{\prime}}(f(x)) ~~\mbox{for all}~x\in L,$
%\item $g(\phi(a))=\psi(g(a))~~\mbox{for all}~a\in A$
%\item $g(\rho_L(x)(a))=\rho_{L^{\prime}}(f(x))(g(a))~~\mbox{for all}~x\in L,~ a\in A.$
%\end{enumerate}  
%\end{Def} 
%Hom-Lie-Rinehart algebras with homomorphisms form a category  of hom-Lie-Rinehart algebras, which we denote by $hLR$. Note that the category of Lie-Rinehart algebras is a full subcategory of the category of hom-Lie-Rinehart algebras. 

%\subsection{Homology and Cohomology of hom-Lie-Rinehart algebras}
%Let $A$ be an associative and commutative $R$-algebra and $\phi$ be an algebra homomorphism of $A$ and $(\mathcal{L},\alpha)$ be a hom-Lie-Rinehart algebra over $(A, \phi)$. Let us recall the notion of a left (and right) $(\mathcal{L},\alpha)$-module.

\begin{Def}\label{Mod}
Let $M$ be an $A$-module, and $\beta\in End_{R}(M)$. Then the pair $(M,\beta)$ is a left module over a hom-Lie Rinehart algebra $(\mathcal{L},\alpha_L)$ if the following holds.
\begin{enumerate}[label=(\roman*)]
\item There is a map $\theta:L\otimes M\rightarrow M$, such that the pair  $(\theta,\beta)$ is a representation of the hom-Lie algebra $(L,[-,-]_L,\alpha_L)$ on $M$. Let us denote $\theta(x,m)$ by $\{x,m\}$ for $x\in L,~m\in M$.
\item $\beta(a.m)=\phi(a).\beta(m)$ for all $a\in A~\mbox{and}~m\in M$.
\item $\{a.X,m\}=\phi(a)\{X,m\}$ for all $a\in A,~X\in L,~m\in M$.
\item $\{X,a.m\}=\phi(a)\{X,m\}+\rho_L(X)(a).\beta(m)$ for all $X\in L,~a\in A,~m\in M$.
\end{enumerate} 
\end{Def}
In particular, for $\alpha_L=Id_L$ and $\beta=Id_M$, $(\mathcal{L},\alpha_L)$ is a Lie-Rinehart algebra and $M$ is a left Lie-Rinehart algebra module over the Lie-Rinehart algebra $L$.
\begin{Exm}
The pair $(A,\phi)$ is a canonical left $(\mathcal{L},\alpha_L)$-module, where the left action of $L$ on $A$ is given by the anchor map. 
\end{Exm}
Let $(\mathcal{L},\alpha_L)$ be a hom-Lie Rinehart algebra over $(A,\phi)$ and $(M,\beta)$ be a left module over $(\mathcal{L},\alpha_L)$. Let us consider the graded $R$-module 
$$C^*(L;M):=\oplus_{n\geq 1}C^n(L;M)$$
where, $C^n(L;M)\subseteq Hom_R(\wedge_R^n L,M)$ consisting of elements $f\in Hom_R(\wedge_R^n L,M)$ satisfying the following conditions.
\begin{enumerate}
\item $f(\alpha_L(x_1),\cdots,\alpha_L(x_n))=\beta(f(x_1,x_2,\cdots,x_n))$ for all $x_i\in L,~1\leq i\leq n$
\item $f(x_1,\cdots,a.x_i,\cdots,x_n)=\phi^{n-1}(a)f(x_1,\cdots,x_i,\cdots,x_n)$ for all $x_i\in L,\\~1\leq i\leq n,~\mbox{and}~ a\in A$.
\end{enumerate}
%for $n\geq 1$.
Define the $R$-linear maps $\delta:C^n(L;M)\rightarrow C^{n+1}(L;M) $ given by 
\begin{equation*}
\begin{split}
&\delta f(x_1,\cdots,x_{n+1})\\
&= \sum_{i=1}^{n+1}(-1)^{i+1}\{\alpha_L^{n-1}(x_i),f(x_1,\cdots,\hat{x_i},\cdots,x_{n+1})\}\\&+\sum_{1\leq i<j\leq n+1}f([x_i,x_j],\alpha_L(x_1),\cdots,\hat{\alpha_L(x_i)},\cdots,\hat{\alpha_L(x_j)},\cdots,\alpha_L(x_{n+1}))
\end{split}
\end{equation*}
for all $f\in C^n(L;M),~ x_i\in L, $ and $1\leq i\leq n+1$. Here, $(C^*(L,M),\delta)$ forms a cochain complex, see \cite{MM2018} for more details. 
%\newpage
\section{Deformation complex for hom-Lie-Rinehart algebras}
In this section, we construct a deformation complex which encodes all the information about deformation of hom-Lie-Rinehart algebras. 

Let $M$ be an ${A}$-module, $\phi:A\rightarrow  A$ be an algebra homomorphism, and  $\beta:M\rightarrow M $ be a $\phi$-function linear map, i.e $\beta(a.x)=\phi(a).\beta(x)$ for $a\in A$, and $x\in M$. A $(\phi,\beta)$-derivation of $M$ is a linear  map $D: M \rightarrow M$ such that there exists $\sigma_D \in Der_{R} (\mathcal{A})$ satisfying the following conditions.
\begin{enumerate}[label=(\roman*)]
\item $D (f x) = f D (x) + \sigma_D (f) x, ~\text{ for } ~ f \in \mathcal{A}, x \in M.$
\item $D\circ \beta=\beta\circ D$, and $\sigma_D\circ\phi=\phi\circ \sigma_D$.
\end{enumerate}
 
Let us denote by $Der_{\phi} (M,\beta)$, the space of $(\phi,\beta)$-derivations of $M$. It is a Lie algebra with the Lie bracket, given by the commutator bracket. 

Next, we consider a graded $R$-module of multiderivations on which we describe a graded Lie algebra structure by extending the canonical Lie algebra structure on the space of $(\phi,\beta)$-derivations of $M$.

\begin{Def}
Let $M$ be an $A$-module, $\phi:A\rightarrow  A$ be an algebra homomorphism, and  $\beta:M\rightarrow M $ be a $\phi$-function linear map. Then a linear map 
$$\textstyle{D:\wedge^{n+1}M\rightarrow M}$$
is called a $(\phi, \beta)$-multiderivation of degree $n$ (of the $A$-module $M$) if there exists a linear map $\sigma_D:\wedge^{n}M\rightarrow Der_{\phi^n}A$ such that the following conditions are satisfied.
\begin{enumerate}[label=(\roman*)]
\item $\textstyle{D(\beta(x_1),\beta(x_2),\cdots, \beta(x_{n+1}))=\beta(D(x_1,x_2,\cdots,x_{n+1}))},$
\item  $\textstyle{\sigma_D(\beta(x_1),\beta(x_2)\cdots, \beta(x_n))(\phi(a))=\phi(\sigma_D(x_1,x_2, \cdots,x_n)(a))},$
\item $\textstyle{\sigma_D(x_1,x_2,\cdots,a.x_n)=\phi^{n}(a).\sigma_D(x_1,x_2,\cdots,x_{n})},$ and
\item $\textstyle{D(x_0,x_1,\cdots,a.x_{n})=\phi^n(a)D(x_0,\cdots,x_{n})+\sigma_D(x_0,\cdots,x_{n-1})(a).\beta^n(x_n)}$,
\end{enumerate}
for all $x_0,\cdots,x_n\in M$, and $a\in A$.

   The map $\sigma_D$ is called the symbol map of the $(\phi, \beta)$-multiderivation $D$. Let us denote the space of n-degree $(\phi,\beta)$-multiderivations of $M$ by $\mathfrak{Der}^n_{\phi}(M,\beta)$. 
\end{Def}

\begin{Rem} Let us consider the following cases:
\begin{enumerate}[label=(\alph*)]
\item If $n=0$, then $\mathfrak{Der}^n_{\phi}(M,\beta)= Der_{\phi}(M,\beta)$, the space of $(\phi,\beta)$-derivations of $A$-module $M$.
\item For $\phi=id_A$ and $\beta=id_M$, the above definition describes the notion of $A$-module multiderivations \cite{LUCA15} of an $A$-module $M$ (with a change in degree convention).
\item  In particular, for a vector bundle over a smooth manifold $N$, by setting $A=C^{\infty}(N)$ and $M=\Gamma E$, we get the multiderivations of the vector bundle $E$ (defined in \cite{CM08}), i.e. for $n\geq 0$
$$\mathfrak{Der}^n_{\phi}(M,\beta)= Der^n(E).$$
\end{enumerate}
\end{Rem}
Next, we extend the Lie bracket of $Der_{\phi}(M,\beta)$ to a graded Lie bracket on the space of $(\phi,\beta)$-multiderivations of $M$
$$\mathfrak{Der}^*_{\phi}(M,\beta):=\oplus_{n \geq 0} \mathfrak{Der}^n_{\phi}(M,\beta).$$
Let $D_1\in \mathfrak{Der}^p_{\phi}(M,\beta)$, and $D_2\in \mathfrak{Der}^q_{\phi}(M,\beta)$, then define a bracket as follows:
\begin{equation}\label{GLB}
[D_1, D_2] := (-1)^{pq} D_1 \circ D_2 - D_2 \circ D_1.
\end{equation}
In the above expression, the product $D_1\circ D_2$ is given by the expression below for any $x_0,\cdots,x_p,\cdots,x_{p+q}$.
\begin{equation}\label{Defofcirc}
\begin{split}
&(D_1 \circ D_2)(x_0, x_1, \ldots, x_{p+q})\\ &= \textstyle{\sum\limits_{\tau \in Sh (q+1, p)} (-1)^{|\tau|}
D_1 \big(D_2 (x_{\tau(0)}, \ldots, x_{\tau (q)}),\beta^q(x_{\tau (q+1)}) , \ldots, \beta^q(x_{\tau (p+q))}) \big)},
\end{split}
\end{equation}
Here, by $Sh (q+1, p)$ we denote the $(q+1,p)$ shuffles in $S_{q+p+1}$ (the symmetric group on the set $\{1,\cdots,p+q+1\}$), and for any permutation $\tau\in S_{q+p+1}$, $|\tau|$ denotes the signature of the permutation $(\tau)$ .

It follows that the bracket $[-,-]:\mathfrak{Der}^*_{\phi}(M,\beta)\times \mathfrak{Der}^*_{\phi}(M,\beta)\rightarrow \mathfrak{Der}^*_{\phi}(M,\beta)$ is a graded Lie-bracket of degree $0$ (see \cite{AEM11} for details). 
For any $a\in A,$ and $x_0,\cdots,x_{p+q}\in L$, let us consider the following expression
$$[D_1,D_2](x_0, x_1, \ldots, a.x_{p+q})=\Big((-1)^{pq}D_1\circ D_2-D_2\circ D_1\Big)(x_0, x_1, \ldots, a.x_{p+q}).$$
Let us denote by $Sh^1(q+1,p)$, the set of $(q+1,p)$-shuffles of the set $\{0,1,\cdots,p+q\}$ fixing $p+q$, and denote by $Sh^2(q+1,p)$, the set of $(q+1,p)$-shuffles, which send $q\mapsto p+q$. Clearly, $Sh^1(q+1,p)\cup Sh^2(q+1,p)= Sh(q+1,p)$. Then using equation \eqref{Defofcirc}, we get the following equation:
\begin{equation}\label{P1}
\begin{split}
&(D_1 \circ D_2)(x_0, x_1, \ldots, a.x_{p+q})\\ 
&= \textstyle{\sum\limits_{\tau \in Sh^1 (q+1, p)} (-1)^{|\tau|}
D_1 \big(D_2 (x_{\tau(0)}, \ldots, x_{\tau (q)}),\beta^q(x_{\tau (q+1)}) , \ldots, \beta^q(x_{\tau(p+q-1)}),\beta^q(a.x_{(p+q)}) \big)}\\
&+\textstyle{\sum\limits_{\tau \in Sh^2 (q+1, p)} (-1)^{|\tau|}
D_1 \big(D_2 (x_{\tau(0)}, \ldots,x_{\tau(q-1)}, a.x_{(p+q)}),\beta^q(x_{\tau (q+1)}) , \ldots, \beta^q(x_{\tau(p+q)}) \big)}\\
%\end{align}
%where $Sh^1(q+1,p)$ is the set of $(q+1,p)$-shuffles of the set $\{0,1,\cdots,p+q\}$ fixing $p+q$, and $Sh^2(q+1,p)$ contains those of $(q+1,p)$-shuffles, which sends $q\mapsto p+q$. Clearly, $Sh^1(q+1,p)\cup Sh^1(q+1,p)=Sh(q+1,p)$. We can rewrite \eqref{P1} as follows
%\begin{align}\label{Istpart}
%&(D_1 \circ D_2)(x_0, x_1, \ldots, a.x_{p+q})\\ \nonumber
&= \textstyle{\sum\limits_{\tau \in Sh^1(q+1, p)} (-1)^{|\tau|}
\Big(\phi^{p+q}(a).D_1 \big(D_2 (x_{\tau(0)}, \ldots, x_{\tau (q)}),\beta^q(x_{\tau (q+1)}) , \ldots,\beta^q(x_{(p+q)}) \big)}\\
&+\textstyle{\sigma_{D_1} \big(D_2 (x_{\tau(0)}, \ldots, x_{\tau (q)}),\beta^q(x_{\tau (q+1)}) , \ldots, \beta^q(x_{\tau(p+q-1)})\big)(\phi^q(a)). \beta^{p+q}(x_{(p+q)}) \Big)}\\
&+\textstyle{\sum\limits_{\tau \in Sh^2(q+1, p)} (-1)^{|\tau|}
D_1 \big(\phi^q(a).D_2 (x_{\tau(0)}, \ldots,x_{(p+q)}),\beta^q(x_{\tau (q+1)}) , \ldots, \beta^q(x_{\tau(p+q)}) \big)}\\
&+\textstyle{\sum\limits_{\tau \in Sh^2(q+1, p)} (-1)^{|\tau|}
D_1 \big(\sigma_{D_2} (x_{\tau(0)}, \ldots,x_{\tau(q-1)})(a).\beta^q(x_{p+q}),\beta^q(x_{\tau (q+1)}) , \ldots, \beta^q(x_{\tau(p+q)}) \big)}
%&=A+B+C.
\end{split}
\end{equation}
%where, 
%\begin{align*}
%A
%&   = \sum_{\tau \in Sh^1} (-1)^{|\tau|}
%\Big(\phi^{p+q}(a).D_1 \big(D_2 (x_{\tau(0)}, \ldots, x_{\tau (q)}),\beta^q(x_{\tau (q+1)}) , \ldots,\beta^q(x_{(p+q)}) \big)\\\nonumber
%&+\sigma_{D_1} \big(D_2 (x_{\tau(0)}, \ldots, x_{\tau (q)}),\beta^q(x_{\tau (q+1)}) , \ldots, \beta^q(x_{\tau(p+q-1)})\big)(\phi^q(f)). \beta^{p+q}(x_{(p+q)}) \Big)\\
%&\\
%B&=\sum_{\tau \in Sh^2} (-1)^{|\tau|}
%D_1 \big(\phi^q(a).D_2 (x_{\tau(0)}, \ldots,x_{(p+q)}),\beta^q(x_{\tau (q+1)}) , \ldots, \beta^q(x_{\tau(p+q)}) \big)\\
%&\\
%C&:=\sum_{\tau \in Sh^2} (-1)^{|\tau|}
%D_1 \big(\sigma_{D_2} (x_{\tau(0)}, \ldots,x_{\tau(q-1)})(a).\beta^q(x_{p+q}),\beta^q(x_{\tau (q+1)}) , \ldots, \beta^q(x_{\tau(p+q)}) \big)
%\end{align*}
%Next, $B,$ and $C$ can further be written as

Let us denote the set $Sh^2(q+1,p)$ simply by $Sh^2$. The expression in the second summation appeared on the right hand side of equation \eqref{P1} can be written as $B_1+B_2$, where 
\begin{align*}
B_1&=\textstyle{\sum\limits_{\tau \in Sh^2 } (-1)^{|\tau|}
\phi^{p+q}(a).D_1 \big(D_2 (x_{\tau(0)}, \ldots,x_{(p+q)}),\beta^q(x_{\tau (q+1)}) , \ldots, \beta^q(x_{\tau(p+q)}) \big)},\\
B_2&=\textstyle{\sum\limits_{\tau \in Sh^2} (-1)^{|\tau|+p}\sigma_{D_1}(\beta^q(x_{\tau (q+1)}) , \ldots, \beta^q(x_{\tau(p+q)})(\phi^q(a)).\beta^p(D_2 (x_{\tau(0)}, \ldots,x_{(p+q)}))}.
\end{align*}

Next, the expression in the third summation appeared on the right hand side of equation \eqref{P1} can be written as $C_1+C_2$, where
\begin{align*}
C_1&=\textstyle{\sum\limits_{\tau \in Sh^2 } (-1)^{|\tau|}\phi^p(\sigma_{D_2} (x_{\tau(0)}, \ldots,x_{\tau(q-1)})(a)).D_1\big(\beta^q(x_{p+q}),\beta^q(x_{\tau (q+1)}) , \ldots, \beta^q(x_{\tau(p+q)}) \big)},\\
C_2&=\textstyle{\sum\limits_{\tau \in Sh^2 } (-1)^{|\tau|+p}\sigma_{D_1}\big(\beta^q(x_{\tau (q+1)}) , \ldots, \beta^q(x_{\tau(p+q)}) \big)\sigma_{D_2} (x_{\tau(0)}, \ldots,x_{\tau(q-1)})(a).\beta^{p+q}(x_{p+q}).}
\end{align*}

Similarly, 
\begin{equation}\label{IIndpart}
\begin{split}
&\textstyle{(D_2 \circ D_1)(x_0, x_1, \ldots, a.x_{p+q})}\\ 
&= \textstyle{\sum\limits_{\tau \in Sh^1(p+1,q)} (-1)^{|\tau|}
\Big(\phi^{p+q}(a).D_2 \big(D_1 (x_{\tau(0)}, \ldots, x_{\tau (p)}),\beta^p(x_{\tau (p+1)}) , \ldots,\beta^p(x_{(p+q)}) \big)}\\
&+\textstyle{\sigma_{D_2} \big(D_1 (x_{\tau(0)}, \ldots, x_{\tau (p)}),\beta^p(x_{\tau (p+1)}) , \ldots, \beta^p(x_{\tau(p+q-1)})\big)(\phi^p(a)). \beta^{p+q}(x_{(p+q)}) \Big)}\\
&+\textstyle{\sum\limits_{\tau \in Sh^2(p+1,q)} (-1)^{|\tau|}
D_2 \big(\phi^p(a).D_1 (x_{\tau(0)}, \ldots,x_{\tau(p-1)},x_{(p+q)}),\beta^p(x_{\tau (p+1)}) , \ldots, \beta^p(x_{\tau(p+q)}) \big)}\\
&+\textstyle{\sum\limits_{\tau \in Sh^2(p+1,q)} (-1)^{|\tau|}
D_2 \big(\sigma_{D_1} (x_{\tau(0)}, \ldots,x_{\tau(p-1)})(a).\beta^p(x_{p+q}),\beta^p(x_{\tau (p+1)}) , \ldots, \beta^p(x_{\tau(p+q)}) \big)}\\
%&=I+J+K.
\end{split}
\end{equation}

Let us denote $\textstyle{Sh^2(p+1,q)}$ simply by $\bar{Sh^2}$. The expression in the second summation appeared on the right hand side of equation \eqref{IIndpart} can be written as $J_1+J_2$, where
\begin{equation*}
\begin{split}
J_1&=\textstyle{\sum\limits_{\tau \in \bar{Sh^2}} (-1)^{|\tau|}
\phi^{p+q}(a).D_2 \big(D_1 (x_{\tau(0)}, \ldots,x_{(p+q)}),\beta^p(x_{\tau (p+1)}) , \ldots, \beta^p(x_{\tau(p+q)}) \big)}\\
J_2&=\textstyle{\sum\limits_{\tau \in \bar{Sh^2}} (-1)^{|\tau|+q}\phi^p\big(\sigma_{D_2}(x_{\tau (q+1)} , \ldots, x_{\tau(p+q)})(a)\big).\beta^q\big(D_1(x_{\tau(0)}, \ldots,x_{\tau(p-1)},x_{(p+q)})\big)} \\
\end{split}
\end{equation*}

The expression in the third summation appeared on the right hand side of equation \eqref{IIndpart} can be written as
\begin{equation*}
\begin{split}
K_1&=\textstyle{\sum\limits_{\tau \in \bar{Sh^2}} (-1)^{|\tau|}\phi^q(\sigma_{D_1} (x_{\tau(0)}, \ldots,x_{\tau(p-1)})(a)).D_2\big(\beta^p(x_{p+q}),\beta^p(x_{\tau (p+1)}) , \ldots, \beta^p(x_{\tau(p+q)}) \big)}\\
K_2&=\textstyle{\sum\limits_{\tau \in \bar{Sh^2}} (-1)^{|\tau|+q}\sigma_{D_2}\big(\beta^p(x_{\tau (p+1)}) , \ldots, \beta^p(x_{\tau(p+q)}) \big)\sigma_{D_1} (x_{\tau(0)}, \ldots,x_{\tau(p-1)})(a).\beta^{p+q}(x_{p+q})}
\end{split}
\end{equation*}
%Here, $\bar{Sh}^1=Sh^1(p+1,q),$ and $\bar{Sh}^2=Sh^2(p+1,q)$.

Next, using the properties of multiderivations $D_1, D_2$ and there symbols $\sigma_{D_1}, \sigma_{D_2}$, let us observe the following:
\begin{enumerate}
\item $(-1)^{pq}B_2-K_1=0$, and 
\item $(-1)^{pq}C_1-J_2=0.$
\end{enumerate} 
Consequently, using equations \eqref{P1} and \eqref{IIndpart}, we get the following identity:
\begin{align*}
&\textstyle{[D_1,D_2](x_0,\cdots,x_{p+q-1},a.x_{p+q})}\\&
=\textstyle{\phi^{p+q}(a).[D_1,D_2](x_0,x_1,\cdots,x_{p+q-1},x_{p+q})
+\sigma_{[D_1,D_2]}(x_0,x_1,\cdots,x_{p+q-1})(a).\beta^{p+q}(x_{p+q})}.
\end{align*}

In the above identity, the map 
$\sigma_{[D_1, D_2]}:\wedge^{p+q}M\rightarrow Der_{\phi^{p+q}}(A)$
is given by the following equation
\begin{equation}\label{symbofpr}
\sigma_{[D_1, D_2]} = (-1)^{pq} \sigma_{D_1} \odot D_2 - \sigma_{D_2} \odot D_1 + \{\sigma_{D_1}, \sigma_{D_1}\}.
\end{equation}

The terms appeared in the right hand side of the above equation are given as follows. For any $x_1,\ldots,x_{p+q}\in M$, and $a\in A$,

\begin{enumerate}[label=(\roman*)]
\item $\textstyle{\{\sigma_{D_1}, \sigma_{D_1}\}}$ is defined by 
\begin{align*}
&\textstyle{\{\sigma_{D_1}, \sigma_{D_2}\}(x_1, \ldots, x_{p+q})(a)}\\& = \textstyle{\sum\limits_{\tau \in Sh(p,q)} (-1)^{|\tau|} \Big(\sigma_{D_1}\big(\beta^q(x_{\tau(1)}), \ldots, \beta^q(x_{\tau (p)})\big) \sigma_{D_2} (x_{\tau(p+1)}, \ldots, x_{\tau (p+q)})}\\
&-\sigma_{D_2}\big(\beta^p(x_{\tau(p+1)}), \ldots, \beta^p(x_{\tau (p+q)})\big) \sigma_{D_1} (x_{\tau(1)}, \ldots, x_{\tau (p)})\Big)(a),\\
\end{align*}

\item  $\sigma_{D_1}\odot D_2$ is defined by 
\begin{align*}
&\textstyle{(\sigma_{D_1} \odot D_2)(x_1, \ldots, x_{p+q})(a)}\\ &= \textstyle{\sum\limits_{\tau \in Sh (q+1, p-1)} (-1)^{|\tau|}
\sigma_{D_1} \big(D_2 (x_{\tau(1)}, \ldots, x_{\tau (q+1)}),\beta^q(x_{\tau (q+2)}) , \ldots, \beta^q(x_{\tau (p+q))} \big)(\phi^q(a))}.
\end{align*}
\end{enumerate} 

Hence, for $\textstyle{D_1\in \mathfrak{Der}^p_{\phi}(M,\beta)}$, and $\textstyle{D_2\in \mathfrak{Der}^q_{\phi}(M,\beta)}$, the bracket $\textstyle{[D_1,D_2]\in \mathfrak{Der}^{p+q}_{\phi}(M,\beta)}$ with the symbol map $\sigma_{[D_1,D_2]}$. Therefore the space of $(\phi,\beta)$-multiderivation $ \mathfrak{Der}^*_{\phi}(M,\beta)$ is closed under the graded Lie bracket given by equation \eqref{GLB}.
\begin{Thm}\label{GLA}
Let $M$ be an $A$-module and $\textstyle{\beta:M\rightarrow M}$ be a $\phi$-function linear map. Then the space of $(\phi,\beta)$-multiderivations of $M$ has a graded Lie algebra structure, where the graded Lie bracket is given by equation \eqref{GLB}.
\end{Thm}

In the next result, we describe a hom-Lie-Rinehart algebra structures in terms of the graded Lie algebra obtained above.
\begin{Prop}\label{LR-1der}
Let $L$ be an ${A}$-module and $\textstyle{\alpha_L:L\rightarrow L}$ be a $\phi$-function linear map. Then there is a one-to-one correspondence between hom-Lie-Rinehart algebra structures on the pair $(L,\alpha_L)$ and elements $\textstyle{\mathfrak{m}\in \mathfrak{Der}^1_{\phi}(L,\alpha_L)}$ satisfying $\textstyle{[\mathfrak{m},\mathfrak{m}]=0}$.
\end{Prop}
\begin{proof}
Let $\textstyle{(\mathcal{L},\alpha_L)}$ be a hom-Lie-Rinehart algebra over $(A,\phi)$. Let us define a bilinear map $\mathfrak{m} : L \times L \rightarrow L$ by $\mathfrak{m}(x, y) := [x, y]_L$, for $x,y \in L$. By definition for any $x,y\in L$ and $a\in A$, we get 
\begin{equation}\label{id1}
\textstyle{\mathfrak{m}(x,a.y)=\phi(a).\mathfrak{m}(x,y)+\rho_L(x)(a).\alpha_L(y)}.
\end{equation}

By equation \eqref{id1}, it follows that $\mathfrak{m}$ is a $1$-degree $(\phi,\alpha_L)$-derivation of the $A$-module $L$, i.e. $\mathfrak{m} \in \mathfrak{Der}^1_{\phi}(L,\alpha_L)$ with symbol $\sigma_{\mathfrak{m}}=\rho_L:L\rightarrow Der_{\phi}(A)$. Moreover, from the definition of the graded Lie bracket \eqref{GLB} it is clear that
\begin{align*}
[\mathfrak{m}, \mathfrak{m}](x,y,z) &= -2 (\mathfrak{m} \circ \mathfrak{m}) (x,y,z) \\
&= \textstyle{2 \bigg{\{}[\alpha_L(x),[y,z]_L]_L + [\alpha_L(y), [z,x]_L]_L + [\alpha_L(z), [x,y]_L]_L\bigg{\}}}=0.
\end{align*}

Conversely, let us start with an element $\mathfrak{m} \in \mathfrak{Der}^{1}_{\phi}(L,\alpha_L)$ (with symbol $\sigma_{\mathfrak{m}}$) satisfying the identity: $[\mathfrak{m},\mathfrak{m}] = 0$. Let us define a bracket 
$[-,-]_L:L\times L\rightarrow L$ as follows
 $$[x,y]_L=\mathfrak{m}(x,y)$$
  for any $x,y\in L$. Also, define a linear map $\rho_L:=\sigma_{\mathfrak{m}}:L\rightarrow Der_{\phi}(A)$. Then it follows that $(A, L,[-,-]_L,\phi,\alpha_L,\rho_L)$ is a hom-Lie-Rinehart algebra over $(A,\phi)$. 
\end{proof}

Thus, we have a description of hom-Lie-Rinehart algebra structures on the pair $(L,\alpha_L)$ in terms of the graded Lie algebra structure on $(\phi,\alpha_L)$-multiderivations of $L$, obtained by Theorem \ref{GLA}.

\subsection{Deformation complex}\label{subsection-def-complex}
Let $(\mathcal{L},\alpha_L)$ be a hom-Lie-Rinehart algebra over $(A,\phi)$. Then by Proposition \ref{LR-1der}, this hom-Lie-Rinehart algebra structure on $(L,\alpha_L)$ corresponds to an element $\mathfrak{m}\in \mathfrak{Der}^1_{\phi}(L,\alpha_L)$ such that $[\mathfrak{m},\mathfrak{m}] = 0.$

Let us define a cochain complex $(C^*_{def}(\mathcal{L},\alpha_L),\delta)$, where 
$$ \textstyle{C^*_{def}(\mathcal{L},\alpha_L) :=\oplus_{n\geq 1}~ C^n_{def}(\mathcal{L},\alpha_L)},~~\mbox{and}~~~ \textstyle{C^n_{def}(\mathcal{L},\alpha_L):= \mathfrak{Der}^{n-1}_{\phi}(L,\alpha_L)}.$$ 
Here, the differential $$\delta : C^n_{def}(\mathcal{L},\alpha_L) \rightarrow C^{n+1}_{def}(\mathcal{L},\alpha_L)$$ is given by
$$\delta (D) = [\mathfrak{m},D],$$ for $D \in C^n_{def}(\mathcal{L},\alpha_L)$. In particular, for any $x_0,\cdots, x_{n}\in L$, and $D\in \mathfrak{Der}^{n-1}_{\phi}(L,\alpha_L)$, the coboundary expression is given as follows.
\begin{equation} \label{defor-diff-1}
\begin{split}
&\delta(D) (x_0,x_1, \ldots,x_n) \\
=& \sum_{i=1}^n (-1)^i \mathfrak{m}\big(\alpha_L^{n-1}(x_i), D(x_0, \ldots, \widehat{x}_i, \ldots, x_n)\big) \\
& + \sum_{0 \leq i \leq j \leq n} (-1)^{i+j} D\big(\mathfrak{m}(x_i, x_j), \alpha_L(x_0), \ldots, \widehat{X}_i, \ldots, \widehat{X}_j, \ldots, \alpha_L(x_n)\big). 
\end{split}
\end{equation}
Next, let us observe that $\mathfrak{m}\in \mathfrak{Der}^1_{\phi}(L,\alpha_L)$ satisfies $[\mathfrak{m},\mathfrak{m}]=0$, therefore graded Jacobi identity for the graded Lie bracket $[-,-]:\mathfrak{Der}^*_{\phi}(L,\alpha_L)\times \mathfrak{Der}^*_{\phi}(L,\alpha_L)\rightarrow \mathfrak{Der}^*_{\phi}(L,\alpha_L)$ implies that $\delta^2=0$. 
\begin{Prop}
Let $(\mathcal{L},\alpha_L)$ be a hom-Lie-Rinehart algebra over $(A,\phi)$ and  $\mathfrak{m}$ be the corresponding $1$-degree $(\phi,\alpha_L)$-derivation satisfying $[\mathfrak{m},\mathfrak{m}]=0$. Then with the coboundary operator $\delta : C^n_{def} (L) \rightarrow C^{n+1}_{def} (L)$ defined by $\delta(D) = [\mathfrak{m},D]$, for $D\in C^n_{def} (L,\alpha_L)$, the cochain complex $(C^*_{def} (\mathcal{L},\alpha_L) , \delta)$ is a differential graded Lie algebra (DGLA with a shift in degree). 
\end{Prop}

Let us denote by $H^*_{def} (\mathcal{L},\alpha_L)$, the cohomology of the cochain complex $(C^*_{def}(\mathcal{L},\alpha_L), \delta)$. Next, we show that the cohomology $H^*_{def} (\mathcal{L},\alpha_L)$ is the {\sf deformation cohomology} of the hom-Lie-Rinehart algebra $(\mathcal{L},\alpha_L)$.

%\begin{Exm}
%Recall that if $\phi:A\rightarrow A$ is an algebra automorphism then the tuple $(\mathfrak{Der}_{\phi}(A), Ad_{\phi}):=(A,Der_{\phi}(A),[-,-]_{\phi},\phi,Ad_{\phi},Ad_{\phi})$ is a hom-Lie-Rinehart algebra structure $($over $(A,\phi))$ on the space of $\phi$-derivations of the commutative, associative $R$-algebra $A$. Let
%$(D, \sigma_D) \in C^n_{def} \big(Der_{\mathbb{K}} (\mathcal{A}) \big)$ be such that
%$\delta (D, \sigma_D) = (\delta(D), \sigma_{\delta(D)}) = 0$. Thus it follows from (\ref{defor-diff-2}) that
%$(-1)^{n-1} D = \delta(\sigma_D).$
%This shows that any cocycle in $C^n_{def} ( Der_{\mathbb{K}} (\mathcal{A}))$ is also exact, that is, $H^\bullet_{def} ( Der_{\mathbb{K}} (\mathcal{A})) = 0.$
%\end{Exm}

%\newpage
\section{Deformation of Hom-Lie-Rinehart algebras}
In view of Proposition \ref{LR-1der}, we now consider the hom-Lie-Rinehart algebra structure on $(L,\alpha_L)$ over $(A,\phi)$ as an element $\mathfrak{m}\in \mathfrak{Der}^1_{\phi}(L,\alpha_L)$ satisfying $[\mathfrak{m},\mathfrak{m}]=0$. Here, we denote by $R[[t]]$, the space of formal power series ring with parameter $t$. 
\begin{Def}
  A deformation  of a hom-Lie-Rinehart algebra structure on $(L,\alpha_L)$ (over $(A,\phi)$) which is given via $\mathfrak{m}\in \mathfrak{Der}^1_{\phi}(L,\alpha_L)$, is a $R[[t]]$-bilinear map 
$$\mathfrak{m}_t:L[[t]] \otimes L[[t]]\rightarrow L[[t]]$$ 
which is $\mbox{given by}~\mathfrak{m}_t(x,y)=\sum_{i\geq0} t^i m_i(x,y) ~\mbox{for} ~~m_0=\mathfrak{m}~ \mbox{and}~ m_i\in \mathfrak{Der}^1_{\phi}(L,\alpha_L)~ \mbox{for}~ i> 0,$ and 
satisfy $[[\mathfrak{m}_t,\mathfrak{m}_t]]=0$. Here, $[[-,-]]$ is a graded Lie algebra bracket on $\mathfrak{Der}^*_{\phi_t}(L[[t]],\alpha_{L_t})$ and the map $\phi_t:A[[t]]\rightarrow A[[t]]$ is an algebra homomorphism extending the algebra homomorphism $\phi:A\rightarrow A$ such that $\phi_t(t)=t$. 
\end{Def}

\begin{Rem}
Note that $\mathfrak{m}_t=\sum_{i\geq0} t^i m_i$ is a 1-degree $(\phi_t,\alpha_t)$ multiderivation of $L[[t]]$ with the symbol $\sigma_{\mathfrak{m}_t}$ given by 
$$\sigma_{\mathfrak{m}_t}(x):=\sum_{i\geq0} t^i \sigma_{m_i}(x):L[[t]]\rightarrow Der_{\phi_t}\big(A[[t]]\big).$$
Since $\mathfrak{m}_t$ satisfies $[[\mathfrak{m}_t,\mathfrak{m}_t]]=0,$ it corresponds to a hom-Lie-Rinehart algebra structure on $(L[[t]],\alpha_t)$ over $\big(A[[t]],\phi_t\big)$.
\end{Rem}

If $D \in ker(\delta_1)$ then $\delta(D)=[\mathfrak{m},D]=0~\mbox{for}~D\in Der_{\phi}(L,\alpha_L)$. i.e.,
$$
\mathfrak{m}(D(x),y)+\mathfrak{m}(x,D(y))= D(\mathfrak{m}(x,y)).
$$
Therefore, $H^1_{def}$ gives set of all $(\phi,\alpha_L)$-derivations on hom-Lie-Rinehart algebra $(\mathcal{L},\alpha_L)$.

Let $\mathfrak{m}_t$ be a deformation of $\mathfrak{m}$. Then 
$$\mathfrak{m}_t(\alpha_L(a),\mathfrak{m}_t(b,c))+\mathfrak{m}_t(\alpha_L(b),\mathfrak{m}_t(c,a))+\mathfrak{m}_t(\alpha_L(c),\mathfrak{m}_t(a,b))=0.$$
Comparing the coefficients of $t^n$ for $n\geq 0$, we get the following equations:
\begin{equation}\label{DefEqu}
\sum_{i,j=0}^n m_i(\alpha_L(a),m_j(b,c))+m_i(\alpha_L(b),m_j(c,a))+m_i(\alpha_L(c),m_j(a,b))=0.
\end{equation}
For $n=1$, equation \eqref{DefEqu} implies $[\mathfrak{m},m_1]=\delta(m_1)=0$, i.e. $m_1$ is a 2-cocycle.

The 2-cochain $m_1$ is called the infinitesimal of the deformation $\mathfrak{m}_t$. More generally, if $m_i=0$ for $1\leq i\leq (n-1)$ and $m_n$ is non zero cochain then $m_n$ is called the $n$-infinitesimal of the deformation $\mathfrak{m}_t$. By the above discussion the following proposition immediately follows.

\begin{Prop}
The infinitesimal of the deformation $\mathfrak{m}_t$ is a 2-cocycle in $C^2_{def}$. More generally, the $n$-infinitesimal is a 2-cocycle. 
\end{Prop}
\begin{Def}
Two deformations $\mathfrak{m}_t$ and $\tilde{\mathfrak{m}_t}$ are said to be equivalent if we have a formal automorphism $$\Phi_t : L[[t]] \rightarrow L[[t]]~~\mbox{ defined as}~~
\Phi_t=id_L+\sum_{i\geq 1}t^i\phi_i$$
where for each $i\geq 1$, $\phi_i:L \rightarrow L$ is a $R$-linear map such that 
$$\phi_i\circ \alpha_L=\alpha_L\circ \phi_i ~~\mbox{and}~
\tilde{\mathfrak{m}_t}(x,y)=\Phi_t^{-1}\mathfrak{m}_t(\Phi_tx,\Phi_ty).
$$
\end{Def}
\begin{Def}
Any deformation equivalent to the deformation $m_0=\mathfrak{m}$ is said to be a trivial deformation.
\end{Def}
\begin{Thm}\label{Eq1}
The cohomology class of the infinitesimal of a deformation $\mathfrak{m}_t$ is determined by the equivalence class of $\mathfrak{m}_t$.
\end{Thm}

\begin{proof}
Let $\Phi_t$ represents an equivalence of deformation given by $\mathfrak{m}_t $ and $ \tilde{\mathfrak{m}_t}$. Then we get,
$$
\tilde{\mathfrak{m}_t}(x,y)=\Phi_t^{-1}\mathfrak{m}_t(\Phi_tx,\Phi_ty).
$$
Expanding the above identity and comparing the coefficients of $t$ from both sides of the above equation we have 
$$
m_1-\tilde{m_1}=[\mathfrak{m},\phi_1].
$$
So, cohomology class of infinitesimal of the deformation is determined by the equivalence class of deformation of $\mathfrak{m}_t$.
\end{proof}
\begin{Def}
A hom-Lie-Rinehart algebra is said to be rigid if and only if every deformation of it is trivial.
\end{Def}

\begin{Thm}
A non-trivial deformation of a hom-Lie-Rinehart algebra is equivalent to a deformation whose n-infinitesimal is not a coboundary for some $n\geq 1$.
\end{Thm}

\begin{proof}
Let $\mathfrak{m}_t$ be a deformation of hom-Lie-Rinehart algebra with n-infinitesimal $m_n$ for some $n\geq 1$. Assume that there exists a 2-cochain $\phi \in C^1_{def}(\mathcal{L},\alpha_L)$ with $\delta(\phi)=m_n$. Then set
$$
\Phi_t=id_{L}+\phi t^n ~~\mbox{and define}~~\bar{\mathfrak{m}_t}=\Phi_t \circ \mathfrak{m}_t \circ \Phi_t^{-1}.
$$
Then by computing the expression and comparing coefficients of $t^n$, we get
$$
\bar{m_n}-m_n=-[\mathfrak{m},m_n]=-\delta(\phi).
$$
So, $\bar{m_n}=0$.
We can repeat the argument to kill off any infinitesimal, which is a coboundary.
\end{proof}

\begin{Cor}
If $H^2_{def}(\mathcal{L},\alpha_L)=0$, then hom-Lie-Rinehart algebra is rigid.
\end{Cor}
\subsection{An example of a rigid  Hom-Lie-Rinehart algebra}
Here we deduce that the space of $\phi$-derivations of an associative commutative algebra is a rigid Hom-Lie-Rinehart algebra.

 Let $A$ be an associative commutative $R$-algebra and $\phi:A\rightarrow A$ be an algebra automorphism, then the space of all $\phi$-derivations of $R$-algebra $A$, denoted by $Der_{\phi}(A)$ has a hom-Lie-Rinehart algebra structure.  Let us recall that  
  $$(\mathcal{D}er_{\phi}(A), Ad_{\phi}):=(A,Der_{\phi}(A),[-,-]_{\phi},\phi,Ad_{\phi},Ad_{\phi})$$ is a hom-Lie-Rinehart algebra over $(A,\phi)$.

Let $D\in \mathfrak{Der}^1_{\phi}(\mathcal{D}er_{\phi}(A), Ad_{\phi})$ with symbol $\sigma_D$. Let $D$ be a $2$-cocycle, i.e. $$\delta(D)=[m,D]=0,$$ 
where $m=[-,-]_{\phi}$, and therefore it follows that 
$\sigma_{[m,D]}=0$. Now from equation \eqref{symbofpr}, we have the following relation 
$$ \textstyle{0=\sigma_{[m, D]} = - \sigma_{m} \odot D - \sigma_{D} \odot m + \{\sigma_{m}, \sigma_{D}\}}.$$
Here, $\sigma_m=Ad_{\phi}$, and for any $x,y\in Der_{\phi}(A),$ and $a\in A$,  
\begin{align*}
\textstyle{\sigma_{m} \odot D(x,y)(a)}&=\textstyle{Ad_{\phi}(D(x,y))(\phi(a))},\\
\textstyle{\{\sigma_{m}, \sigma_{D}\}(x,y)(a)}&=\textstyle{Ad_{\phi}^2(x)\sigma_D(y)(a)-\sigma_D(Ad_{\phi}(y))Ad_{\phi}(x)(a)}\\
&\textstyle{-Ad_{\phi}^2(y)\sigma_D(x)(a)+\sigma_D(Ad_{\phi}(x))Ad_{\phi}(y)(a)},\\
\textstyle{(\sigma_{D} \odot m)(x,y)(a)}&=\textstyle{\sigma_D([x,y]_{\phi})(\phi(a))}
\end{align*}

Therefore, we get the following equation
\begin{align}\label{part1}
{\phi}(D(x,y)(a))&=-\sigma_D([x,y]_{\phi})(\phi(a))+Ad_{\phi}^2(x)\sigma_D(y)(a)\\\nonumber
&-\sigma_D(Ad_{\phi}(y))Ad_{\phi}(x)(a)-Ad_{\phi}^2(y)\sigma_D(x)(a)\\\nonumber
&+\sigma_D(Ad_{\phi}(x))Ad_{\phi}(y)(a).
\end{align}

Let us consider, a map $Ad_{\phi}^{-1}\circ\sigma_D:Der_{\phi}(A)\rightarrow Der_{\phi}(A)$, then clearly it is an $A$-linear map and it commutes with the map $Ad_{\phi}:Der_{\phi}(A)\rightarrow Der_{\phi}(A)$. Therefore, $Ad_{\phi}^{-1}\circ\sigma_D$ is a zero-degree $(\phi,Ad_{\phi})$-derivation of the pair ${D}er_{\phi}(A)$. 
\begin{align}\label{part2}
\nonumber
\delta(Ad_{\phi}^{-1}\circ\sigma_D)(x,y)(a)&=[m,Ad_{\phi}^{-1}\circ\sigma_D]
(x,y)(a)\\
&=[Ad_{\phi}^{-1}\circ\sigma_D(x),y]_{\phi}(a)-[Ad_{\phi}^{-1}\circ\sigma_D(y),x]_{\phi}(a)\\\nonumber
&-Ad_{\phi}^{-1}\circ\sigma_D([x,y]_{\phi})(a).
\end{align}

The first two terms on the right-hand side of the previous expression can be written as follows:
\begin{align}\label{part3}
\nonumber
&\Big([Ad_{\phi}^{-1}\circ\sigma_D(x),y]_{\phi}-[Ad_{\phi}^{-1}\circ\sigma_D(y),x]_{\phi}\Big)(a)\\\nonumber
=&\sigma_D(x)(y(\phi^{-1}(a))-\phi(y(\phi^{-2}(\sigma_D(x)(a))))\\\nonumber
-&\sigma_D(y)(x(\phi^{-1}(a))
+\phi(x(\phi^{-2}(\sigma_D(y)(a))))\\
=&\phi^{-1}\Big(\sigma_D(Ad_{\phi}(x))Ad_{\phi}(y)(a)-Ad_{\phi}^2(y)\sigma_D(x)(a)\\\nonumber
-&\sigma_D(Ad_{\phi}(y))Ad_{\phi}(x)(a)
+Ad_{\phi}^2(x)\sigma_D(y)(a)\Big).
\end{align}

By using equations \eqref{part1}-\eqref{part3}, we deduce that 
$$D(x,y)(a)=[m,Ad_{\phi}^{-1}\circ\sigma_D]
(x,y)(a)=\delta(Ad_{\phi}^{-1}\circ\sigma_D)(x,y)(a).$$

Therefore, any $2$-cocyle is a $2$-coboundary, i.e. $H^2_{def}(\mathcal{D}er_{\phi}(A), Ad_{\phi})=0$. Consequently, we have the following proposition 

\begin{Prop}
The hom-Lie-Rinehart algebra $(\mathcal{D}er_{\phi}(A), Ad_{\phi})$ is rigid.
\end{Prop}

\subsection{Obstructions to the extension of deformations}
Let $(\mathcal{L},\alpha_L)$ be a hom-Lie-Rinehart algebra and the hom-Lie-Rinehart structure on the pair $(L,\alpha_L)$ corresponds to $\mathfrak{m}\in C^2_{def}(\mathcal{L},\alpha_L)$. Now we consider the problem of extending a deformation of $\mathfrak{m}$ of order $N$ to a deformation of $\mathfrak{m}$ of order $(N+1)$. Let $m_t$ be a deformation of $\mathfrak{m}$ of order $N$. Then $$\mathfrak{m}_t=\sum^N_{i=0}m_i t^i$$
where $m_i\in C^2_{def}(\mathcal{L},\alpha_L)$ for each $1\leq i\leq N$ such that $[[\mathfrak{m}_t,\mathfrak{m}_t]]=0$. If there exists a 2-cochain $m_{N+1}\in C^2_{def}(\mathcal{L},\alpha_L)$ such that $\tilde{\mathfrak{m}_t}=\mathfrak{m}_t+m_{N+1} t^{N+1}$ is a deformation of $\mathfrak{m}$ of order $N+1$. Then we say that $\mathfrak{m}_t$ extends to a deformation of order $(N+1)$.

\begin{Def}
A deformation of $\mathfrak{m}$ of order $N$ is given by $\textstyle{\mathfrak{m}_t=\sum^N_{i=0}m_i t^i}$ modulo $t^{N+1}$, such that $[[\mathfrak{m}_t,\mathfrak{m}_t]]=0$.
If there exists a 2-cochain $m_{N+1}\in C^2_{def}(\mathcal{L},\alpha_L)$, such that $$\tilde{\mathfrak{m}_t}=\mathfrak{m}_t+m_{N+1} t^{N+1}$$ is a deformation of order $N+1$. Then we say that $\mathfrak{m}_t$ extends to a deformation of order $(N+1)$.
\end{Def}

\begin{Def}
Let $\mathfrak{m}_t$ be a deformation of $\mathfrak{m}$ of order $N$. Consider the cochain \\
$~~~~~~~~~\Theta(\mathcal{L},\alpha_L) \in C^3_{def}(\mathcal{L},\alpha_L)$, where
\begin{equation*}
\begin{split}
\Theta(\mathcal{L},\alpha_L)(a,b,c)
=\sum\limits_{i+j=N+1;~i,j>0} \Big( m_i(\alpha_L(a),m_j(b,c))+&m_i(\alpha_L(b),m_j(c,a))\\
&+m_i(\alpha_L(c),m_j(a,b))\Big).
\end{split}
\end{equation*}
for $a,b,c \in L$. The 3-cochain $\Theta(\mathcal{L},\alpha_L)$ is called the obstruction cochain for extending the deformation of $\mathfrak{m}$ of order $N$ to a deformation of order $N+1$. We can write $\Theta(\mathcal{L},\alpha_L)$ as follows
\begin{equation}\label{Obst}
\Theta(\mathcal{L},\alpha_L)=-1/2\sum_{i+j=N+1;~i,j>0}[m_i,m_j]
\end{equation}
By equation \eqref{Obst}, and using graded Jacobi identity it follows that $\Theta(\mathcal{L},\alpha_L)$ is a 3-cocycle.
\end{Def}

\begin{Thm}\label{hom-Obst}
Let $\mathfrak{m}_t$ be a deformation of $\mathfrak{m}$ of order $N$. Then $\mathfrak{m}_t$ extends to a deformation of order $N+1$ if and only if the cohomology class of $\Theta(\mathcal{L},\alpha_L)$ vanishes.

\begin{proof}
Suppose that a deformation $\mathfrak{m}_t$ of order $N$ extends to a deformation of order $N+1$. Then 
$$
\sum_{i+j=N+1;~ i,j\geq 0} m_i(\alpha_L(a),m_j(b,c))+m_i(\alpha_L(b),m_j(c,a))+m_i(\alpha_L(c),m_j(a,b))=0.
$$
As a result, we get $\Theta(\mathcal{L},\alpha_L)=\delta(m_{N+1})$. So, cohomology class of $\Theta(\mathcal{L},\alpha_L)$ vanishes.

Conversely, let $\Theta(\mathcal{L},\alpha_L)$ be a coboundary. Suppose that
$$
\Theta(\mathcal{L},\alpha_L)=\delta(m_{N+1}) 
$$
for some 2-cochain $m_{N+1}$. Define a map $\tilde{\mathfrak{m}_t}:L[[t]]\times L[[t]]\rightarrow L[[t]]$ as follows
$$
\tilde{\mathfrak{m}_t}=\mathfrak{m}_t+m_{N+1}t^{N+1}.
$$
Then for any $x,y,z\in L $, the map $\tilde{\mathfrak{m}_t}$ satisfies the following identity
$$
\sum_{i+j=N+1;~ i,j\geq 0} m_i(\alpha_L(x),m_j(y,z))+m_i(\alpha_L(y),m_j(z,x))+m_i(\alpha_L(z),m_j(x,y))=0.
$$
 This, in turn, implies that $\tilde{\mathfrak{m}_t}$ is a deformation of $\mathfrak{m}$ extending $\mathfrak{m}_t$.
\end{proof}

\begin{Cor}
If $H^3_{def}(\mathcal{L},\alpha_L)=0$, then every 2-cocycle in $C^2_{def}(\mathcal{L},\alpha_L)$ is an infinitesimal of some deformation of $\mathfrak{m}$.
\end{Cor}
\end{Thm}
%\newpage
\section{Some particular cases}
In this section, we discuss some particular cases of hom-Lie-Rinehart algebras and the associated deformation complex governed by suitable differential graded Lie algebras. 

\subsection{Lie-Rinehart algebras}

Let $(\mathcal{L},\alpha_L)$ be a hom-Lie-Rinehart algebra over $(A,\phi)$. If $\alpha_L=Id_L$, and $\phi=id_A$, then the hom-Lie-Rinehart algebra $(\mathcal{L},\alpha_L)$ is a Lie-Rinehart algebra $L$ over $A$. 

Next, let $L$ be an $A$-module and $\beta:L\rightarrow L $ be a $\phi$-function linear map, then in the case when $\beta=id_{L},$ and $\phi=id_A$, the definition of degree $n$ $(n\geq 0)$ $(\phi,\beta)$-multiderivation of the pair $(L,\beta)$ gives the notion of an $A$-module multiderivation \cite{LUCA15} of degree $n+1$. However, here we consider the following definition of $A$-module multiderivations (with a change in the convention of degree from \cite{LUCA15}).
\begin{Def}
An $A$-module multiderivation (of degree $n\geq 0$) of $L$ is a skew-symmetric $R$-multilinear map:
$$
D:\wedge^{n+1}L\rightarrow L.
$$
for which we have an $A$-multilinear map $\sigma_D \in Hom_A(\wedge^n_AL,Der(A)) $ such that
\begin{equation*}
\begin{split}
&D(x_0,\cdots,f.x_i,\cdots,x_n)\\
&= f.D(x_0,\cdots,x_i,\cdots,x_n)+(-1)^{n-i}\sigma_D(x_0,\cdots,\hat{x_i},\cdots,x_{n-1})(f).x_i.
\end{split}
\end{equation*}

\end{Def}

Denote the space of all n-degree multiderivations as $Der^n(L)$, $n\geq 0$. Set $Der^{-1}(L)=L.$
This gives a graded $A$-module $Der^*(L)=\oplus_{i\in \mathbb{Z}}Der^i(L)$, where $Der^i(L)=0~ \mbox{for}~ i\leq-2$.
Note that for any $n\geq 0$, we have $$Der^n(L)=\mathfrak{Der}^n_{id_A}(L,id_{L}).$$
\begin{Rem}
Let $\mathcal{E}$ be a faithful $A$-module. A Koszul connection on $\mathcal{E}$ is an $A$-linear mapping $\nabla:~Der(A)~\rightarrow~Hom_R(\mathcal{E},\mathcal{E})$, sending $X\mapsto \nabla_X$ such that
$$
\nabla_X(a.m)=X(a).m+a.\nabla_X(m)$$
for $m\in \mathcal{E},~X\in Der(A),$ and $a\in A$. In general, an $A$-module may not admit a Koszul connection. But for any Projective $A$-module, there exists a Koszul connection.
\end{Rem}
Following lemma generalises its geometric counterpart for Lie algebroid case, given in \cite{CM08}:
\begin{Lem}
Space of $A$-module multiderivations of degree n on $L$, i.e. $Der^n(L)$ fits into an exact sequence of $A$-modules:
$$
0\rightarrow Hom_A(\wedge^{n+1}_AL,L) \rightarrow Der^n(L) \rightarrow Hom_A(\wedge^{n}_AL,Der(A)) \rightarrow 0,
$$
for $n\geq0$.
\end{Lem}

\begin{proof}
Define, $F_D \in Hom_R(\otimes^{n+1}L,L)$ by
$$
F_D(x_0,\cdots,x_n)=D(x_0,\cdots,x_n)+(-1)^n \sum_{i=0}^n(-1)^{i+1}\nabla_{\sigma_D(x_0,\cdots,\hat{x_i},\cdots,x_n)}(x_i).
$$
It follows that $F_D$ is skew-symmetric and $A$-multilinear. Also, one can check that a connection $\nabla$ on $L$ determines an isomorphism of $A$-modules
$$
Der^n(L) \cong Hom_A(\wedge^{n+1}_A L,L) \oplus Hom_A(\wedge^{n}_AL,Der(A))
$$
assigning $D~\rightarrow~(F_D,\sigma_D).$

\end{proof}

\subsubsection{Graded Lie algebra structure on $Der^{*}(L)$.}
From Theorem \ref{GLA}, it is immediate that there is a graded Lie algebra structure on the graded $A$-module $Der^{*}(L)$. If $D_1 \in Der^{p}(L)$, and  $D_2 \in Der^{q}(L)$, then the graded Lie bracket is defined as follows:
$$ [D_1,D_2]=(-1)^{pq}D_1\circ D_2~-~D_2\circ D_1 $$
where,
\begin{equation*}
\begin{split}
&D_1\circ D_2(x_0,\cdots,x_{p+q})\\
=&\sum_{\tau \in S(q+1,p)}sgn(\tau)D_1(D_2(x_{\tau(0)},\cdots,x_{\tau(q+1)}),\cdots,x_{\tau(p+q)}).
\end{split}
\end{equation*}
It follows that $[D_1,D_2]\in Der^{p+q}(L)$ and
$$
\sigma_{[D_1,D_2]}=(-1)^{pq} \sigma_{D_1}\circ D_2-\sigma_{D_2}\circ D_1+[\sigma_{D_1},\sigma_{D_2}].
$$
where,
\begin{equation*}
\begin{split}
&[\sigma_{D_1},\sigma_{D_2}](x_1,\cdots,x_{p+q})\\
=&\sum_{\tau \in S(p,q)}sgn(\tau)[\sigma_{D_1}(x_{\tau(1)},\cdots,x_{\tau(p)}),\sigma_{D_2}(x_{\tau(p+1)},\cdots,x_{\tau(p+q)})]
\end{split}
\end{equation*}

Next result is a particular case of Proposition \ref{LR-1der}, which describes Lie-Rinehart algebra structures on $L$ in term of the above-mentioned graded Lie algebra.

\begin{Prop}
If $L$ is an $A$-module, then there exists a one-one correspondence between Lie-Rinehart algebra structures on $L$ and elements $m\in Der^1(L)$ satisfying $[m,m]=0$.
\end{Prop}

\subsubsection{Deformation complex for Lie-Rinehart algebra}
Let $L$ be a Lie-Rinehart algebra over $A$, then it corresponds to an element $m\in Der^1(L)$ satisfying $[m,m]=0$. We define $$C^n_{def}(L):=C^n_{def}(\mathcal{L},id_L)$$
for $n\geq 1,$ and $C^0_{def}(L)=Der^{-1}(L)=L$. The coboundary $\delta:C^n_{def}(L) \rightarrow C^{n+1}_{def}(L)$ is given by $\delta(D)=[m,D]$. Thus, $(C^*_{def}(L),\delta)$ is a differential graded Lie algebra and it forms a cochain complex. We denote the cohomology of this cochain complex by $H^*_{def}(L)$.

\begin{Exm}
In particular, for Lie algebra $\mathfrak{g}$ the deformation complex $C^*_{def}(\mathfrak{g})$ is the usual  Chevalley-Eilenberg complex $C^*(\mathfrak{g},\mathfrak{g})$ with coefficients in the adjoint representation.
\end{Exm}

\begin{Exm}
For any Lie algebroid, the above deformation complex will be the deformation complex, defined in \cite{CM08}.
\end{Exm}

\begin{Exm}
For $L=Der(A)$, the commutator bracket $[-,-]_C$ gives a Lie algebra bracket and by considering the anchor map $\rho=id$, $(L,[-,-],\rho)$ is a Lie-Rinehart algebra. Consider $Der(A)$ to be a projective $A$-module. then
$$
Z^k(C^*_{def}(L)) = Hom_A(\wedge^{k-1}L,L).
$$
$\delta(D)=0$. Now, as we know
$$
\sigma_{\delta(D)}=\delta(\sigma_D)+(-1)^{k-1}\rho~o~D
$$
i.e., $D=(-1)^k\delta(\sigma_D)$.
Therefore,
$$
H^*_{def}(Der(A))=0.
$$
\end{Exm}

\begin{Def}
 A deformation  of a Lie-Rinehart algebra structure on $L$, which is given via $m\in Der^1(L)$, is defined as a $R[[t]]$-bilinear map $m_t:L[[t]] \otimes L[[t]]\rightarrow L[[t]]$, given by 
$$m_t[x,y]=\sum_{i\geq0} t^i m_i[x,y], ~\mbox{where} \quad m_0=m~ \mbox{and}~ m_i\in Der^1(L)~ \mbox{for}~ i\geq 0$$
satisfying $[[m_t,m_t]]$=0. Here, $[[-,-]]$ is the graded Lie algebra bracket on ${Der}^*(L[[t]])$.
\end{Def}
\begin{Rem}
If $D \in ker(\delta_1)$, then $\delta(D)=0~\mbox{for}~D\in Der(L)$. i.e.,
$$
m(D(x),y)+m(x,D(y))= D(m(x,y)).
$$
if $D\in Im(\delta_0)$, then D(y)=m(x,y).
So, $H^1_{def}$ gives set of all outer derivation on $L$. 
\end{Rem}
\begin{Rem}
Let $m_t$ be a deformation of $m$. Then we have:
\begin{equation}
m_t(a,m_t(b,c))+m_t(b,m_t(c,a))+m_t(c,m_t(a,b))=0.
\end{equation}
From equation (1), $\delta(m_1)=0$, i.e. $m_1$ is a 2-cocycle.
\end{Rem}

By Theorem \ref{Eq1}, it follows that $H^2_{def}(L)$ characterizes the non-trivial infinitesimal deformations of the hom-Lie-Rinehart algebra $L$. Thus, if $H^2_{def}(L)=0$, then Lie-Rinehart algebra is rigid. Moreover, Theorem \ref{hom-Obst} implies that obstructions to extend a deformation of order $n$ of Lie-Rinehart algebra $L$ to a deformation of order $ n+1$ are contained in the $3$-rd cohomology group $H^3_{def}(L)$. Therefore, $H^*_{def}(L)$ is deformation cohomology for the Lie-Rinehart algebra $L$, obtained as a particular case of hom-Lie-Rinehart algebra.

\begin{Rem}
In \cite{CARA}, the authors have discussed deformations of Lie-Rinehart algebras as an application of the deformation theory of Courant algebroids. In particular, they identified the deformation complex of a Lie-Rinehart algebra in terms of a specific Rothstein algebra.
\end{Rem}

\subsection{Hom-Lie algebras:}
Let $(L,[-,-]_L,\alpha_L)$ be a hom-Lie algebra (over $R$), then it is a hom-Lie module over itself by adjoint action. Recall that $(L,[-,-]_L,\alpha_L)$ is also a hom-Lie-Rinehart algebra $(\mathcal{L},\alpha_L)$ over $(R,id_R)$. Then a $(id_R,\alpha_L)$-multiderivation $\varphi$ of degree $n$ is simply a $n+1$-linear alternating map 
$$\varphi:\wedge^{n+1}L\rightarrow L,$$
satisfying $\alpha_L\circ \varphi=\varphi\circ \alpha_L^{\otimes(n+1)}$. Therefore,  

$$\mathfrak{Der}^n_{id_R}(L,\alpha_L)=C^{n+1}_{HL}(L,L),$$
where, $C^{n+1}_{HL}(L,L)$ is the space of $(n+1)$-linear alternating cochains defined in Section $2$ of \cite{AEM11}. Next, it is easy to see that the differential graded Lie algebra structure on $\mathfrak{Der}^*_{id_R}(L,\alpha_L)$ is also the same as the one discussed in \cite{AEM11}. Consequently, in the case of hom-Lie algebras, the cohomology $H^*_{def}(\mathcal{L},\alpha_L)$ is the same as the deformation cohomology $H^*_{HL}(L,L)$ defined in \cite{AEM11}. 

\subsection{Lie algebroids:}  
A Deformation complex for Lie algebroids is defined by considering the associated differential graded Lie algebra structure on the space of multiderivations in \cite{CM08}. 
On the other hand, one can deduce a deformation complex for Lie algebroids by considering them as a hom-Lie-Rinehart algebra.

 Let $E$ be a Lie algebroid over a smooth manifold $M$ with the anchor map $\rho$ and the Lie-bracket $[-,-]$ on the space of sections $\Gamma E$. Then the Lie algebroid structure on $E$ yields a hom-Lie-Rinehart algebra $(\mathcal{L},\alpha_L)$ over $(A,\phi)$ where $A = C^{\infty}(M),~ \phi=id_A,~ L = \Gamma E, [-,-]_L=[-,-],~\alpha_L = id_L,$ and $\rho_L = \rho$. By considering $\phi=id_A$ and $\beta=id_{L}$, the space of $(\phi,\beta)$-multiderivations turns out to be the space of multiderivations of the vector bundle $E$. Moreover, the differential graded Lie algebra structure on the space of $(\phi,\beta)$-multiderivations is the one (except in degree $-1$)  studied for Lie algebroids. 

\section*{\textbf{Conclusion:}}

Here, we described a differential graded Lie algebra (DGLA) for  hom-Lie-Rinehart algebras, which controls the formal one-parameter deformations. This study goes in line with  particular cases such as hom-Lie algebra, Lie-Rinehart algebras. One can expect to associate such a  differential graded Lie algebra to a hom-Lie algebroid since any hom-Lie algebroid is also a particular type of hom-Lie-Rinehart algebras. Next, it is natural to ask: can we interpret this differential graded Lie algebra in terms of deformations of a hom-Lie algebroid? This type of questions we plan to address in a separate note by introducing deformations of a hom-Lie algebroid as a smooth family of hom-Lie algebroids over an interval $I\subset \mathbb{R}$.

% Moreover, one can interpret the second cohomology of associated cochain complex (to the dgla) of a hom-Lie algebroid in terms of its deformations.  
%\newpage


\begin{thebibliography}{10}

\bibitem{AEM11}
Faouzi Ammar, Zeyneb Ejbehi, and Abdenacer Makhlouf.
\newblock Cohomology and deformations of {H}om-algebras.
\newblock {\em J. Lie Theory}, 21(4):813--836, 2011.

\bibitem{CM08}
Marius Crainic and Ieke Moerdijk.
\newblock Deformations of {L}ie brackets: cohomological aspects.
\newblock {\em J. Eur. Math. Soc. (JEMS)}, 10(4):1037--1059, 2008.

\bibitem{TrMs}
Sam Evens, Jiang-Hua Lu, and Alan Weinstein.
\newblock Transverse measures, the modular class and a cohomology pairing for
  {L}ie algebroids.
\newblock {\em Quart. J. Math. Oxford Ser. (2)}, 50(200):417--436, 1999.

\bibitem{HLS06}
Jonas~T. Hartwig, Daniel Larsson, and Sergei~D. Silvestrov.
\newblock Deformations of {L}ie algebras using {$\sigma$}-derivations.
\newblock {\em J. Algebra}, 295(2):314--361, 2006.

\bibitem{Hueb98}
Johannes Huebschmann.
\newblock Lie-{R}inehart algebras, {G}erstenhaber algebras and
  {B}atalin-{V}ilkovisky algebras.
\newblock {\em Ann. Inst. Fourier (Grenoble)}, 48(2):425--440, 1998.

\bibitem{Hueb99}
Johannes Huebschmann.
\newblock Duality for {L}ie-{R}inehart algebras and the modular class.
\newblock {\em J. Reine Angew. Math.}, 510:103--159, 1999.

\bibitem{Hueb04}
Johannes Huebschmann.
\newblock Lie-{R}inehart algebras, descent, and quantization.
\newblock In {\em Galois theory, {H}opf algebras, and semiabelian categories},
  volume~43 of {\em Fields Inst. Commun.}, pages 295--316. Amer. Math. Soc.,
  Providence, RI, 2004.

\bibitem{CARA}
Frank Keller and Stefan Waldmann.
\newblock Deformation theory of {C}ourant algebroids via the {R}othstein
  algebra.
\newblock {\em J. Pure Appl. Algebra}, 219(8):3391--3426, 2015.

\bibitem{LGT13}
Camille Laurent-Gengoux and Joana Teles.
\newblock Hom-{L}ie algebroids.
\newblock {\em J. Geom. Phys.}, 68:69--75, 2013.

\bibitem{MS10}
Abdenacer Makhlouf and Sergei Silvestrov.
\newblock Notes on 1-parameter formal deformations of {H}om-associative and
  {H}om-{L}ie algebras.
\newblock {\em Forum Math.}, 22(4):715--739, 2010.

\bibitem{MS08}
Abdenacer Makhlouf and Sergei~D. Silvestrov.
\newblock Hom-algebra structures.
\newblock {\em J. Gen. Lie Theory Appl.}, 2(2):51--64, 2008.

\bibitem{MM2017}
Ashis Mandal and Satyendra~Kumar Mishra.
\newblock On {H}om-{G}erstenhaber algebras and {H}om-{L}ie algebroids.
\newblock {\em J. Geom. Phys.}, 133: 287--302, 2018.

\bibitem{MME18}
Ashis Mandal and Satyendra~Kumar Mishra.
\newblock Universal central extensions and non-abelian tensor product of
  {H}om-{L}ie-{R}inehart algebras.
\newblock {\em arXiv:1803.00936v2 [math.KT]}.

\bibitem{MM2018}
Ashis Mandal and Satyendra~Kumar Mishra.
\newblock {H}om-{L}ie-{R}inehart algebras.
\newblock {\em Communications in Algebra}, 46(9):3722--3744, 2018.

\bibitem{Rin63}
George~S. Rinehart.
\newblock Differential forms on general commutative algebras.
\newblock {\em Trans. Amer. Math. Soc.}, 108:195--222, 1963.

\bibitem{Sheng12}
Yunhe Sheng.
\newblock Representations of hom-{L}ie algebras.
\newblock {\em Algebr. Represent. Theory}, 15(6):1081--1098, 2012.

\bibitem{LUCA15}
Luca Vitagliano.
\newblock Representations of homotopy lieâ€“rinehart algebras.
\newblock {\em Mathematical Proceedings of the Cambridge Philosophical
  Society}, 158(1):155--191, 2015.

\bibitem{TFY18}
Tao Zhang, Fengying Han, and Yanhui Bi.
\newblock Crossed modules for {H}om-{L}ie-{R}inehart algebras.
\newblock {\em Colloq. Math.}, 152(1):1--14, 2018.

\end{thebibliography}
\end{document}